\documentclass[11pt]{amsart}

\usepackage[T1]{fontenc}
\usepackage[utf8]{inputenc}
\usepackage{lmodern}
\usepackage{amsmath,amssymb,mathtools,amsthm}
\usepackage{booktabs,array}
\usepackage{xcolor}
\usepackage{microtype}
\usepackage{hyperref}
\usepackage{url}

\hypersetup{
  colorlinks=true,
  linkcolor=black,
  citecolor=black,
  urlcolor=black,
  pdftitle={Determinantal Kernel Schemes of Matrix Nets and Applications to Positive Maps},
  pdfauthor={Trung Hoa Dinh, Minh Toan Ho, Cong Trinh Le, Trung Dung Vuong}
}
\newcommand{\chg}[1]{#1}
\newcommand{\newrev}[1]{#1}

\newcommand{\Gr}{\operatorname{Gr}}
\newcommand{\rank}{\operatorname{rank}}
\newcommand{\Tr}{\operatorname{Tr}}
\newcommand{\Ad}{\operatorname{Ad}}
\newcommand{\SN}{\operatorname{SN}}
\newcommand{\Span}{\operatorname{span}}
\newcommand{\sat}{\mathrm{sat}}
\newcommand{\dd}{\mathrm d}

\theoremstyle{plain}
\newtheorem{theorem}{Theorem}[section]
\newtheorem{lemma}[theorem]{Lemma}
\newtheorem{proposition}[theorem]{Proposition}
\newtheorem{corollary}[theorem]{Corollary}
\theoremstyle{definition}

\theoremstyle{remark}
\newtheorem{remark}[theorem]{Remark}

\title[Determinantal Kernel Schemes of Matrix Nets]{Determinantal Kernel Schemes of Matrix Nets and Applications to Positive Maps}

\author{Trung Hoa Dinh}
\address{Department of Mathematics, Troy University, Troy, AL 36082, USA}
\email{thdinh@troy.edu}

\author{Minh Toan Ho}
\address{Institute of Mathematics, Vietnam Academy of Science and Technology (VAST), 18 Hoang Quoc Viet, Hanoi, Vietnam}
\email{hmtoan@math.ac.vn}

\author{Cong Trinh Le}
\address{Department of Mathematics and Statistics, Quy Nhon University, 170 An Duong Vuong, Quy Nhon, Vietnam}
\email{lecongtrinh@qnu.edu.vn}
\thanks{Corresponding author: lecongtrinh@qnu.edu.vn}

\author{Trung Dung Vuong}
\address{High School for the Gifted, VNUHCM, 153 Nguyen Chi Thanh, An Dong Ward, Ho Chi Minh City, Vietnam; Vietnam National University Ho Chi Minh City, Linh Xuan Ward, Ho Chi Minh City, Vietnam}
\email{vtdung@ptnk.edu.vn}

\date{September 2026}

\begin{document}

\begin{abstract}
{
For a linear map on a matrix space, we study the projective schemes obtained by intersecting its projectivized kernel with determinantal rank loci. For matrix nets, that is, \chg{three dimensional} matrix spaces, we classify every \chg{positive dimensional} intersection with the \chg{rank one} Segre variety in arbitrary rectangular size. The possibilities are a ruling plane, a smooth conic, two Segre lines from opposite rulings, a reduced Segre line, or a Segre line with one reduced or embedded residual point. The smooth conic case is contained in a $2\times2$ compression. For $a,b\ge3$, the corresponding locus in $\Gr(3,M_{a,b})$ has exactly three irreducible components, whose geometry and intersections are determined explicitly. For nets in $M_3(\mathbb C)$, every finite \chg{rank one} scheme has length at most three; \newrev{the adjugate identity gives an intrinsic determinantal obstruction to the length four case allowed for general systems of plane quadrics.} As an application, a \chg{four parameter} family arising from the merging construction admits exact positivity and decomposability criteria. After normalization, positivity is the unit square and decomposability is the quarter disk. The remaining region is atomic and carries explicit PPT entangled states of birank $(5,5)$ and Schmidt number two. \newrev{The point of this application is that the exact phase boundary is realized on a smooth conic determinantal kernel stratum selected independently by the projective classification.}
}
\end{abstract}
\maketitle

\noindent\chg{\textbf{2020 Mathematics Subject Classification.} Primary 14M12, 15A69; Secondary 14M15, 47B65, 81P40.}

\smallskip
\noindent\chg{\textbf{Keywords.} Matrix nets, determinantal varieties, Segre varieties, positive maps, PPT entanglement.}

\section{Introduction and main results}

Positive maps on matrix algebras are central in operator theory and in the dual description of quantum entanglement. Since Choi's indecomposable map on $M_3(\mathbb C)$, many families have been obtained by changing diagonal parameters or cyclic structure, by merging known maps, or by optimizing entanglement witnesses; see \cite{Bera2025,Choi1975,ChoKyeLee1992,ChruscinskiMarciniakRutkowski2018,Kye1992,MarciniakRutkowski2017,MillerOlkiewicz2015}. A complementary polynomial viewpoint encodes a positive map through its Hermitian biquadratic Choi polynomial; recent applications to indecomposable maps, PPT entanglement, and edge states are developed in \cite{HoLeLeOsaka2026}. The main question here is geometric and precedes positivity: what information is carried by the matrices of small rank inside a kernel?

Write $M_{a,b}=M_{a,b}(\mathbb C)$ and $M_n=M_{n,n}$. For a nonzero linear subspace $K\subset M_{a,b}$, write $\mathbb P(K)$ for its projectivization. When $\dim K=3$, the plane $\mathbb P(K)\simeq\mathbb P^2$ is traditionally called a matrix net. The geometry uses the linear matrix space and its determinantal equations, not matrix multiplication. Thus matrix space refers to an arbitrary linear subspace, matrix net to a space of dimension three, and matrix algebra only to the square algebra $M_n$ in the application to positive maps.

\chg{The rank one matrices form the Segre variety $\mathbb P^{a-1}\times\mathbb P^{b-1}\subset\mathbb P(M_{a,b})$. Hence a matrix net produces a plane section of a fixed projective variety, and its rank one scheme records both the set of rank one matrices in the net and possible nonreduced intersection structure. This reformulation makes it possible to use projective geometry before imposing any positivity assumption on a linear map.}

This viewpoint differs from \chg{the theory of matrix spaces of bounded rank}, where every element of a linear space satisfies a rank bound. Here a general element of $K$ may have maximal rank; only the incidence of $\mathbb P(K)$ with determinantal loci is retained. It is also closely related to tensor geometry. If $T\in\mathbb C^3\otimes\mathbb C^a\otimes\mathbb C^b$ has first flattening
\[
T:(\mathbb C^3)^*\longrightarrow\mathbb C^a\otimes\mathbb C^b
\]
of rank three, then $K=T((\mathbb C^3)^*)$ is a matrix space of dimension three and $\mathbb P(K)$ is the corresponding matrix net. In the $3\times3\times3$ case, the base locus studied by Gesmundo and Keneshlou is precisely the \chg{rank one} scheme of this net \cite{GesmundoKeneshlou2026}. Their main object is the associated collineation variety and \chg{the stratification by tensor orbits}; here the \chg{rank one} scheme itself, including embedded structure and its Grassmannian parameter geometry, is the primary object. Related preserver and bounded rank viewpoints appear in \cite{GesmundoHanLovitz2025,HuangLandsberg2026}.

Two features distinguish this setting from these neighboring viewpoints. First, the adjugate identity \newrev{provides an intrinsic determinantal explanation for the absence of} the \chg{length four base locus stratum} that is available for a general plane quadratic system; \newrev{the argument does not rely on a tensor orbit classification.} Second, recent \chg{exact threshold} results for sparse qutrit maps \chg{\cite{PoderiniEtAl2026}} \newrev{and the anisotropic bistochastic family of Sacchi \cite{Sacchi2026}} show that sparsity alone can produce sharp semialgebraic phase boundaries, including square positivity regions, circular decomposability thresholds, and explicit PPT entangled states. \newrev{Accordingly, our positive map contribution is not the mere occurrence of such a phase diagram: its distinguishing feature is that the boundary is realized on the smooth conic component of an independently classified determinantal kernel locus.}

\chg{We first classify the positive dimensional plane sections and then determine their global parameter spaces.}

\begin{theorem}\label{thm:main-positive-dimensional}
Let $a,b\ge2$, let $K\subset M_{a,b}$ have dimension three, and put
\[
Z_K=\mathbb P(K)\cap\Sigma_{a,b},\qquad
\Sigma_{a,b}=\mathbb P^{a-1}\times\mathbb P^{b-1}\subset\mathbb P(M_{a,b}).
\]
If $\dim Z_K\ge1$, then $Z_K$ is exactly one of the following scheme types: a ruling plane, a smooth conic, two reduced Segre lines from opposite rulings, a reduced Segre line, a reduced Segre line with one reduced point off the line, or a Segre line with {one embedded residual point contributing length one}. A doubled Segre line cannot occur outside a ruling linear space. In the \chg{smooth conic} case, $K$ lies in a $2\times2$ compression and is equivalent to the traceless hyperplane of $M_2(\mathbb C)$.
\end{theorem}

For $a,b\ge3$, let $\mathcal L_A$ and $\mathcal L_B$ be the two loci of spaces containing a Segre line of the corresponding ruling, and let $\mathcal C$ be the locus of hyperplanes in a $2\times2$ compression.

\begin{theorem}\label{thm:components}
The locus
\[
\mathcal P_{a,b}=\{K\in\Gr(3,M_{a,b}):\dim(\mathbb P(K)\cap\Sigma_{a,b})\ge1\}
\]
has exactly the three irreducible components $\mathcal L_A$, $\mathcal L_B$, and $\mathcal C$, with dimensions
\[
ab+a+2b-8,\qquad ab+2a+b-8,\qquad 2a+2b-5.
\]
Their pairwise intersections coincide with an irreducible locus of dimension $2a+2b-6$. The compression component is smooth, and this common intersection is its relative Segre divisor.
\end{theorem}

\chg{For $M_3$, the finite part admits a sharper classification than a general quadratic system in a projective plane.}

\begin{theorem}\label{thm:finite-main}
Let $K\subset M_3$ have dimension three. If $Z_K=\mathbb P(K)\cap\Sigma_{3,3}$ is \chg{zero dimensional}, then
\[
\operatorname{length}Z_K\le3.
\]
Up to projective coordinates, the possibilities are the empty scheme, one or two reduced points, one curvilinear double point, three noncollinear reduced points, a curvilinear double point together with a reduced point spanning the plane, \chg{a curvilinear scheme of length three spanning $\mathbb P^2$}, or the planar fat point $\mathfrak m_p^2$. Every type occurs.
\end{theorem}

The significance of the bound is stronger than degree counting: a generic pair of plane quadrics may cut a scheme of length four, whereas the minors of a genuine $3\times3$ linear matrix cannot. The proof in Section~\ref{subsec:finite} identifies the adjugate identity as the mechanism excluding this final case.

\chg{The positive map application is based on the following four parameter family:}
\begin{equation}\label{eq:Psi}
\Psi_{p,q;r,t}(X)=
\begin{pmatrix}
 p(x_{11}+x_{22})&0&r x_{13}\\
 0&p(x_{11}+x_{22})&t x_{32}\\
 r x_{31}&t x_{23}&q x_{33}
\end{pmatrix},
\qquad p,q>0,\quad r,t\ge0.
\end{equation}
It is a scalar specialization of the general merging construction in \cite{MarciniakRutkowski2017} followed by a positive diagonal output congruence, as made explicit in Section~\ref{subsec:positive-family}; the parametrization itself is not claimed as new.

\begin{theorem}\label{thm:main-positive-map}
The map $\Psi_{p,q;r,t}$ is positive if and only if
\[
r^2\le pq,\qquad t^2\le pq.
\]
Inside this positivity region,
\[
\Psi_{p,q;r,t}\text{ is decomposable}
\quad\Longleftrightarrow\quad
r^2+t^2\le pq.
\]
If $r^2+t^2>pq$, the map is atomic. Moreover, \(2\)-positivity is equivalent to complete positivity and to $t=0$, while \(2\)-copositivity is equivalent to complete copositivity and to $r=0$. For {$r,t>0$} the kernel is the \chg{smooth conic} normal form from Theorem~\ref{thm:main-positive-dimensional}.
\end{theorem}

The \chg{four parameter} family admits a \chg{two parameter} normalization. With
\[
R=\frac r{\sqrt{pq}},\qquad T=\frac t{\sqrt{pq}},
\]
an invertible positive output congruence sends $\Psi_{p,q;r,t}$ to $\Psi_{1,1;R,T}$. Hence positivity is the unit square in the first quadrant, decomposability is the quarter disk $R^2+T^2\le1$, and the remaining positive region is atomic. Moreover, the PPT operator \chg{adapted to the parameters} and used to prove atomicity is itself a PPT entangled state of birank $(5,5)$, with Schmidt number two both before and after partial transpose. Thus the \chg{smooth conic determinantal stratum} supports an exactly solvable positive map phase geometry together with an explicit \chg{family of PPT entangled states}.

\chg{At the normalized boundary point $(R,T)=(1,1)$, the construction is congruent to the Miller and Olkiewicz map described in \cite{MillerOlkiewicz2015,MarciniakRutkowski2017}. The criteria above therefore recover its positivity and atomicity. They do not assert the stronger extremality properties established for that map by other methods.}

The remainder of the paper proves these statements. Section~\ref{sec:consequences} records consequences for established Choi constructions and discusses related questions in positive map theory; those comments are not used in the proofs above.

\medskip
\noindent\textbf{Notation.}
All vector spaces, varieties, and schemes are over $\mathbb C$. For a \chg{finite dimensional} vector space $E$, $\mathbb P(E)$ denotes the projective space of \chg{one dimensional} subspaces. We set $\mathbb P(0)=\varnothing$. For $1\le r<\min\{a,b\}$, set
\[
D_r(a,b)=\{[X]\in\mathbb P(M_{a,b}):\rank X\le r\}.
\]
If $L:M_{a,b}\to W$ is linear, its $r$th determinantal kernel scheme is
\[
K_r(L)=\mathbb P(\ker L)\cap D_r(a,b),
\]
with the scheme structure induced by the determinantal ideal. The sequence $(K_r(L))_r$ is the determinantal kernel profile. Unless stated otherwise, intersections with determinantal varieties carry their induced scheme structures. Grassmannian parameter loci, and intersections between these parameter loci, are understood with their reduced induced structures. This convention does not change the induced scheme structures of the determinantal kernel sections.

{
For a matrix net $K$, choose homogeneous coordinates $[s:t:u]$ on $\mathbb P(K)$ and put $R=\mathbb C[s,t,u]$. We distinguish the ideal of restricted minors from the saturated ideal of the projective section:
\[
\begin{aligned}
I_{\min}(K)&=(\text{restricted }2\times2\text{ minors})\subset R,\\
I(Z_K)&=I_{\min}(K)^\sat
       =I_{\min}(K):(s,t,u)^\infty.
\end{aligned}
\]
Thus $Z_K=\operatorname{Proj}(R/I_{\min}(K))=\operatorname{Proj}(R/I(Z_K))$, but the two homogeneous ideals need not be equal. Throughout, $I(Z_K)$ denotes the saturated ideal.
}

\section{Determinantal kernel geometry of matrix nets}\label{sec:geometry}

\subsection{Determinantal equivalence}

Fix $r$ and consider two linear maps $L_1,L_2:M_{a,b}\to W$. \chg{We call them determinantally equivalent at rank $r$} if
\[
L_2=S\circ L_1\circ T
\]
for some $S\in\mathrm{GL}(W)$ and $T\in\mathrm{GL}(M_{a,b})$ whose projectivization preserves $D_r(a,b)$. If the same $T$ preserves every rank locus, we simply call the maps determinantally equivalent.

\begin{proposition}\label{prop:det-equivalence}
If $L_1$ and $L_2$ are \chg{determinantally equivalent at rank $r$}, then $K_r(L_1)$ and $K_r(L_2)$ are projectively isomorphic. Consequently their Hilbert polynomials, dimensions, degrees, \chg{data on irreducible components}, and singularity types agree.
\end{proposition}

\begin{proof}
If $L_2=S\circ L_1\circ T$ and $S$ is invertible, then $\ker L_2=T^{-1}(\ker L_1)$. Since $\mathbb P(T^{-1})$ preserves $D_r(a,b)$,
\[
K_r(L_2)=\mathbb P(T^{-1})(K_r(L_1)).
\]
\end{proof}

{For $r=1$, the classical theorem of Marcus and Moyls says that an invertible linear preserver of rank one matrices on $M_{a,b}$ has the form $X\mapsto PXQ$, with $P\in\mathrm{GL}_a(\mathbb C)$ and $Q\in\mathrm{GL}_b(\mathbb C)$, or, only when $a=b$, the form $X\mapsto PX^{\mathsf T}Q$ \cite{MarcusMoyls1959}. For rectangular matrices, transposition instead interchanges the tensor factors and maps $M_{a,b}$ to $M_{b,a}$.} Modern tensor preserver results place this in the broader geometry of Segre and secant varieties \cite{GesmundoHanLovitz2025}. Determinantal equivalence is broader than the standard equivalences preserving positivity, so a separation obtained from the kernel profile remains valid under unitary changes of basis, invertible congruences, permutations, and transposition symmetries.

\subsection{\chg{Rectangular classification in positive dimension}}\label{subsec:positive-dimensional}

We now prove Theorem~\ref{thm:main-positive-dimensional}. \chg{The first step shows that a smooth conic section forces a $2\times2$ compression.}

\begin{theorem}\label{thm:smooth-conic}
Let $K\subset M_{a,b}$ have dimension three. Suppose that
\[
C\subset\mathbb P(K)\cap\Sigma_{a,b}
\]
is a smooth conic spanning $\mathbb P(K)$ and $\mathbb P(K)\not\subset\Sigma_{a,b}$. Then there exist \chg{two dimensional} subspaces $U\subset\mathbb C^a$ and $V\subset\mathbb C^b$ such that
\[
K\subset U\otimes V\simeq M_2(\mathbb C),
\]
and, \chg{as schemes},
\[
\mathbb P(K)\cap\Sigma_{a,b}
=\mathbb P(K)\cap\bigl(\mathbb P(U)\times\mathbb P(V)\bigr)=C.
\]
Under the natural $\mathrm{GL}(U)\times\mathrm{GL}(V)$ action, $K$ is equivalent to the trace zero hyperplane of $M_2(\mathbb C)$.
\end{theorem}

\begin{proof}
Identify $\Sigma_{a,b}$ with $\mathbb P(A)\times\mathbb P(B)$, where $A\simeq\mathbb C^a$ and $B\simeq\mathbb C^b$, embedded by $\mathcal O(1,1)$. Since $C\simeq\mathbb P^1$ is a conic, the Segre hyperplane bundle has degree two on $C$. If
\[
\alpha=\deg\pi_A^*\mathcal O_{\mathbb P(A)}(1)|_C,
\qquad
\beta=\deg\pi_B^*\mathcal O_{\mathbb P(B)}(1)|_C,
\]
then $\alpha+\beta=2$. Neither degree can be zero: otherwise $C$ would lie in a ruling linear space, and because it spans $\mathbb P(K)$ the whole plane $\mathbb P(K)$ would lie in the Segre variety. Thus $\alpha=\beta=1$.

Each projection maps $C$ isomorphically onto a projective line. Let those lines be $\mathbb P(U)\subset\mathbb P(A)$ and $\mathbb P(V)\subset\mathbb P(B)$. Then
\[
C\subset Q:=\mathbb P(U)\times\mathbb P(V)\subset\mathbb P(U\otimes V)\simeq\mathbb P^3.
\]
Because $C$ spans $\mathbb P(K)$, the plane $\mathbb P(K)$ lies in $\mathbb P(U\otimes V)$. The quadric surface $Q$ meets this plane in a degree two plane section containing $C$, hence the two schemes coincide.

A plane in $\mathbb P(M_2)$ is a hyperplane defined by a nonzero functional $\ell(X)=\Tr(L^{\mathsf T}X)$. Its section of the determinant quadric is smooth exactly when $L$ has rank two. The \chg{left and right action} is transitive on rank two $L$, so $L$ may be normalized to the identity. Thus
\[
K\sim\left\{
\begin{pmatrix}z&x\\y&-z\end{pmatrix}:x,y,z\in\mathbb C
\right\},
\]
whose determinant is $-z^2-xy$.
\end{proof}

\chg{A second step excludes a nonreduced quadratic curve factor unless the entire plane is a ruling plane.}

\begin{lemma}\label{lem:no-double-line}
Let $K\subset M_{a,b}$ have dimension three and put $S=\mathbb P(K)$. If $S\not\subset\Sigma_{a,b}$, then $S\cap\Sigma_{a,b}$ cannot contain a doubled projective line as its \chg{one dimensional} part.
\end{lemma}

\begin{proof}
{We prove the assertion for a Segre line from the first ruling. The second ruling is handled by interchanging the two tensor factors, equivalently by transposing the entire argument and exchanging $a$ and $b$. In the first case, after independent changes of row and column bases, take}
\[
L=\mathbb P\Span\{E_{11},E_{12}\}\subset S,
\qquad
S=\mathbb P\Span\{E_{11},E_{12},A\}.
\]
Write
\[
X(s,t,u)=sE_{11}+tE_{12}+uA.
\]
If $L=\{u=0\}$ occurs with multiplicity at least two, every $2\times2$ minor of $X$ is divisible by $u^2$. For each $i\ge2$, the coefficient of $u$ in the minor on rows $1,i$ and columns $1,2$ is
\[
s a_{i2}-t a_{i1},
\]
so $a_{i1}=a_{i2}=0$. Then for every $j\ge3$, the coefficient of $u$ in the minor on rows $1,i$ and columns $1,j$ is $s a_{ij}$, hence $a_{ij}=0$. Thus every row of $A$ except the first vanishes, and
\[
S\subset\mathbb P(e_1\otimes\mathbb C^b)\subset\Sigma_{a,b},
\]
a contradiction.
\end{proof}

\chg{We can now classify the complete set of positive dimensional plane sections.}

\begin{theorem}\label{thm:positive-plane-sections}
Let $K\subset M_{a,b}$ have dimension three. If $S=\mathbb P(K)$ meets $\Sigma_{a,b}$ in positive dimension, then the underlying set is one of the four alternatives in Theorem~\ref{thm:main-positive-dimensional}: a ruling plane, a smooth conic, two \chg{Segre lines from opposite rulings}, or one Segre line with at most one residual point. The complete scheme structures are precisely the six types in Theorem~\ref{thm:main-positive-dimensional}.
\end{theorem}

\begin{proof}
The ideal of the Segre variety is generated by the $2\times2$ minors. Restrict them to $S\simeq\mathbb P^2$. If every restricted minor vanishes, then $S\subset\Sigma_{a,b}$. A linear space consisting entirely of decomposable tensors has a common factor, so $S$ lies in a ruling linear space.

Assume $S\not\subset\Sigma_{a,b}$. Let ${I=I_{\min}(K)}\subset\mathbb C[s,t,u]$ be generated by the nonzero restricted minors. Every generator is quadratic. If the intersection contains an irreducible curve $g=0$, then $g$ divides every element of $I$, so $\deg g\le2$.

If the greatest common curve factor has degree two, then all nonzero minors are scalar multiples of that quadratic. Over $\mathbb C$ it is either a smooth irreducible conic, a product of two distinct lines, or a square. The square is excluded by Lemma~\ref{lem:no-double-line}; the irreducible case is handled by Theorem~\ref{thm:smooth-conic}. In the reducible case, the two lines lie on the Segre variety. Because two lines in a projective plane intersect, they cannot belong to the same Segre ruling unless their span is a ruling plane contained in the Segre variety. Hence they are from opposite rulings, and their span lies in a $2\times2$ compression as the tangent plane to its Segre quadric.

It remains to suppose the greatest common curve factor is a single linear form $u$. Then the restricted minors have the form $u\ell_j$ with $\ell_j$ linear. If the span of the $\ell_j$ had dimension at most one, there would be a common quadratic factor, already treated. Thus their span has dimension two or three. Their common projective zero set is respectively a single point or the empty set. The scheme refinement is proved in Theorem~\ref{thm:residual-linear-ideal} below.
\end{proof}

\subsection{Global rectangular Grassmannian geometry}

Let $A\simeq\mathbb C^a$, $B\simeq\mathbb C^b$, with $a,b\ge3$. A \chg{Segre line from the first ruling} is $\mathbb P(\mathbb Cu\otimes Z)$ with $[u]\in\mathbb P(A)$ and $Z\in\Gr(2,B)$. Let $\mathcal L_A$ be the locus of matrix nets containing such a line; define $\mathcal L_B$ symmetrically. Let $\mathcal C$ be the locus of \chg{three dimensional} $K$ that are hyperplanes in $U\otimes V$ for some $U\in\Gr(2,A)$ and $V\in\Gr(2,B)$.

\chg{The preceding classification globalizes to three irreducible Grassmannian components.}

\begin{theorem}\label{thm:global-components}
The conclusions and dimension formulas of Theorem~\ref{thm:components} hold.
\end{theorem}

\begin{proof}
The classification in Theorem~\ref{thm:positive-plane-sections} gives
\[
\mathcal P_{a,b}=\mathcal L_A\cup\mathcal L_B\cup\mathcal C.
\]
Lines from the first ruling are parametrized by
\[
F_A=\mathbb P(A)\times\Gr(2,B),
\qquad
\dim F_A=(a-1)+2(b-2)=a+2b-5.
\]
For a fixed line with underlying \chg{two dimensional} vector space $L$, the \chg{three dimensional} subspaces $K\supset L$ are parametrized by
\[
\Gr(1,(A\otimes B)/L)\simeq\mathbb P^{ab-3}.
\]
Thus the incidence variety is projective and irreducible of dimension $ab+a+2b-8$, and its image $\mathcal L_A$ is closed and irreducible. Its projection is generically finite: the embedded $3\times3$ example
\[
K=\Span\{E_{11},E_{12},E_{22}+E_{33}\}
\]
viewed inside $M_{a,b}$ has exactly the displayed Segre line from the first ruling. Hence
\[
\dim\mathcal L_A=ab+a+2b-8.
\]
The other formula is symmetric.

For the compression component, let
\[
B_0=\Gr(2,A)\times\Gr(2,B)
\]
and let $\mathcal U,\mathcal Z$ be the tautological rank two bundles. The projective bundle
\[
X=\mathbb P((\mathcal U\otimes\mathcal Z)^*)\longrightarrow B_0
\]
parametrizes triples $(U,V,[\ell])$ and their hyperplanes $K=\ker\ell\subset U\otimes V$. The induced morphism $X\to\Gr(3,A\otimes B)$ has image $\mathcal C$.

For every such $K$, its left and right supports have dimension exactly two: if, for instance, the left support had dimension one, then $K$ would lie in a \chg{two dimensional} space $\mathbb Cu\otimes V$, impossible because $\dim K=3$. Thus the supports recover $U$ and $V$ uniquely. Indeed, the left support is the image of the contraction map $K\otimes B^*\to A$, and the right support is the image of $K\otimes A^*\to B$. Their ranks are constantly two on $\mathcal C$, so these images define regular \chg{Grassmannian valued} support maps. They give an inverse to $X\to\mathcal C$. Hence
\[
\mathcal C\simeq X,
\qquad
\dim\mathcal C=2(a-2)+2(b-2)+3=2a+2b-5,
\]
and $\mathcal C$ is smooth.

{Let $\mathcal C_{\mathrm{red}}$ denote the locus of reducible conic sections in $\mathcal C$.}
Inside each $\mathbb P^3$ fiber, singular plane sections of $\mathbb P^1\times\mathbb P^1$ are exactly the \chg{rank one} functionals, forming $\mathbb P^1\times\mathbb P^1$. Therefore
\[
{\mathcal C_{\mathrm{red}}}
=\mathbb P(\mathcal U^*)\times_{B_0}\mathbb P(\mathcal Z^*)
\]
is smooth, irreducible, and has dimension $2a+2b-6$. A compression plane contains a Segre line exactly when its quadric section is reducible, so
\[
\mathcal L_A\cap\mathcal C
=\mathcal L_B\cap\mathcal C
={\mathcal C_{\mathrm{red}}}.
\]
If a plane contains one line from each ruling, the two lines intersect and span the tangent plane of the corresponding $2\times2$ Segre quadric. Hence $\mathcal L_A\cap\mathcal L_B={\mathcal C_{\mathrm{red}}}$. Finally, none of the three irreducible loci contains another, as witnessed by a \chg{smooth conic compression plane}, the \chg{single line examples} above, and their transposes.
\end{proof}

\chg{Specializing the dimension formulas immediately gives the case relevant to $M_4$.}

\begin{corollary}\label{cor:M4-components}
For $M_4$, the \chg{locus of rank one schemes of positive dimension} in $\Gr(3,M_4)$ has three irreducible components of dimensions
\[
20,\qquad20,\qquad11,
\]
with common pairwise intersection of dimension $10$.
\end{corollary}

\subsection{Scheme structure along a unique Segre line}

Fix a Segre line from the first ruling and normalize it as
\[
L=\mathbb P\Span\{E_{11},E_{12}\}.
\]
Write $K=\Span\{E_{11},E_{12},A\}$ and, after modifying $A$ by the first two basis elements, assume $a_{11}=a_{12}=0$. In homogeneous coordinates on $\mathbb P(K)$,
\[
X(s,t,u)=sE_{11}+tE_{12}+uA.
\]
Every restricted $2\times2$ minor is divisible by $u$, so
\begin{equation}\label{eq:uJK}
{I_{\min}(K)}
=(u\ell_1,\ldots,u\ell_N)=uJ_K,
\qquad
J_K=(\ell_1,\ldots,\ell_N),
\end{equation}
where $N=\binom a2\binom b2$. {The construction for the second ruling is obtained by transposing the matrices.}

\chg{The residual ideal in \eqref{eq:uJK} is linear and leaves only three possibilities.}

\begin{theorem}\label{thm:residual-linear-ideal}
Assume that $L$ is the unique \chg{one dimensional} component of $\mathbb P(K)\cap\Sigma_{a,b}$. Then $\dim(J_K)_1\ge2$, and exactly one of the following occurs:
\begin{enumerate}
\item $\dim(J_K)_1=3$, in which case {$I(Z_K)=(u)$ and the section is the reduced line $L$};
\item $\dim(J_K)_1=2$ and $u\notin(J_K)_1$, in which case the section is the disjoint union of $L$ and one reduced point off $L$;
\item $\dim(J_K)_1=2$ and $u\in(J_K)_1$, in which case, after a linear change of coordinates preserving $L$,
\[
uJ_K=(u^2,uv)=(u)\cap(u^2,v),
\]
so the section is $L$ with {one embedded residual point whose contribution to the structure sheaf has length one}.
\end{enumerate}
No other scheme structure occurs in the \chg{case of a unique line}.
\end{theorem}

\begin{proof}
{If $\dim(J_K)_1=0$, all restricted minors vanish, so the section is the whole plane, contrary to the assumption that $L$ is a one dimensional component. If $\dim(J_K)_1=1$, write $J_K=(\ell)$ with $\ell\ne0$ linear. When $\ell$ is not proportional to $u$, the ideal $I_{\min}(K)=(u\ell)$ defines two distinct Segre lines, contradicting uniqueness of $L$. When $\ell$ is proportional to $u$, this ideal defines a doubled line, contrary to Lemma~\ref{lem:no-double-line}. Hence $\dim(J_K)_1\ge2$.}

If the dimension is three, $J_K=(s,t,u)$ after a linear change of coordinates, and $(uJ_K)^\sat=(u)$. If the dimension is two, $J_K$ is the homogeneous ideal of a unique point $p\in\mathbb P^2$. When $u\notin(J_K)_1$, one has $p\notin L$ and
\[
uJ_K=(u)\cap J_K.
\]
When $u\in(J_K)_1$, write $J_K=(u,v)$ and obtain
\[
uJ_K=(u^2,uv)=(u)\cap(u^2,v).
\]
{
The ideals in both cases with $\dim(J_K)_1=2$ are saturated, so they equal $I(Z_K)$. To make the embedded contribution precise, put $R=\mathbb C[u,v,w]$. Multiplication by $u$ gives an exact sequence of graded $R$-modules
\[
0\longrightarrow (R/(u,v))(-1)
\xrightarrow{\ \cdot u\ } R/(u^2,uv)
\longrightarrow R/(u)\longrightarrow0.
\]
After sheafification, the kernel of the map to $\mathcal O_L$ is the structure sheaf of the point $p=V(u,v)$, twisted by $-1$, and has length one. Thus the extra embedded structure contributes one to the Hilbert polynomial, giving $d+2$. This length is not the length of the primary component $V(u^2,v)$ in the displayed decomposition, which is two.
}
\end{proof}

\chg{The embedded residual point can also be recognized by tangency to the Segre variety.}

\begin{proposition}\label{prop:tangency}
In case (iii) of Theorem~\ref{thm:residual-linear-ideal}, the embedded point $p\in L$ is characterized by
\[
\mathbb P(K)\subset T_p\Sigma_{a,b}.
\]
Conversely, if a plane contains $L$, is not contained in the Segre variety, and is tangent to $\Sigma_{a,b}$ at $p\in L$ without containing a second Segre line, then its \chg{rank one} section is $L$ with {one embedded residual point at $p$ contributing length one}.
\end{proposition}

\begin{proof}
Write the restricted minors as $q_j=u\ell_j$. At $p\in L$,
\[
\dd q_j(p)=\ell_j(p)\,\dd u.
\]
All differentials vanish on the plane exactly when $\ell_j(p)=0$ for every $j$, that is, when $p=V(J_K)$. This is precisely the embedded case.
\end{proof}

The embedded configuration is nonempty already in $M_3$: with
\[
A_{\mathrm{emb}}=
\begin{pmatrix}
0&0&-1\\
-1&-1&0\\
-1&-1&0
\end{pmatrix},
\qquad
K_{\mathrm{emb}}=\Span\{E_{11},E_{12},A_{\mathrm{emb}}\},
\]
one computes $J_{K_{\mathrm{emb}}}=(t-s,u)$, so the embedded point is $[1:1:0]$.

\subsection{\chg{Hilbert strata and a flat degeneration from a secant configuration to a tangent configuration}}

For the \chg{rank one sections of positive dimension}, the possible Hilbert polynomials are
\[
\binom{d+2}{2},\qquad 2d+1,\qquad d+2,\qquad d+1.
\]
Indeed, they correspond respectively to a ruling plane, a smooth or reducible conic, a line with one residual point, reduced or embedded, and a reduced line.

The equality of the two $d+2$ cases reflects a flat degeneration. In $M_3$, let
\[
K_\varepsilon=\Span\{E_{11},E_{12},E_{13}+E_{21}+\varepsilon E_{23}\}.
\]
Then in coordinates $(s,t,u)$ the nonzero minors generate
\[
I_\varepsilon={I_{\min}(K_\varepsilon)}=(ut,u^2-\varepsilon su)=u(t,u-\varepsilon s).
\]
For $\varepsilon\ne0$, this is a reduced line plus the reduced point $[1:0:\varepsilon]$; at $\varepsilon=0$ it is
\[
(ut,u^2)=(u)\cap(t,u^2),
\]
a line with one embedded point. The graded algebra
\[
R=\mathbb C[\varepsilon,s,t,u]/(ut,u^2-\varepsilon su)
\]
has, for every $d\ge1$, the free $\mathbb C[\varepsilon]$ basis
\[
s^d,s^{d-1}t,\ldots,t^d,s^{d-1}u,
\]
while $R_0=\mathbb C[\varepsilon]$. Hence the family is flat.

\subsection{Finite \chg{rank one} schemes in \texorpdfstring{$M_3$}{M3}}\label{subsec:finite}

We now prove Theorem~\ref{thm:finite-main}. The key point is that a general quadratic system in $\mathbb P^2$ may have a base scheme of length four, as recalled in \cite[Lemma~4.1]{GesmundoKeneshlou2026}, but the quadrics arising as $2\times2$ minors of a $3\times3$ linear matrix satisfy the adjugate identity.

\chg{The next result isolates the determinantal mechanism that forces the sharp bound.}

\begin{theorem}\label{thm:length-bound}
Let $A(s,t,u)$ be a $3\times3$ matrix of linear forms whose coefficient matrices span a \chg{three dimensional} space $K\subset M_3$. If the scheme cut out by the $2\times2$ minors of $A$ in $\mathbb P^2$ is \chg{zero dimensional}, then its length is at most three.
\end{theorem}

\begin{proof}
Let $Z$ be the \chg{rank one} scheme. The homogeneous ideal {$I_{\min}(K)$} generated by the restricted minors has no common curve factor. \newrev{Equivalently, the linear system spanned by these quadrics has no fixed curve component. By the standard genericity argument for plane quadratic systems, as in \cite[Lemma~4.1]{GesmundoKeneshlou2026}, two general quadrics $q_1,q_2$ in this span have no common curve component.} Their complete intersection has length four, so B\'ezout gives
\[
\operatorname{length}Z\le4.
\]
Assume for contradiction that the length is four. Since $Z\subset V(q_1,q_2)$ and both \chg{zero dimensional} schemes have length four, the inclusion is an equality \chg{as schemes}. Moreover, $(q_1,q_2)$ is a saturated \chg{complete intersection} ideal, so
\[
I(Z)_2=\Span\{q_1,q_2\}.
\]
Every entry of $\operatorname{adj}A$ is a $2\times2$ minor. \newrev{Since every such minor belongs to $I(Z)_2=\Span\{q_1,q_2\}$,} there are constant matrices $B_1,B_2$ such that
\[
\operatorname{adj}A=q_1B_1+q_2B_2.
\]
Using
\[
A\operatorname{adj}A=(\det A)I_3
\]
and comparing \chg{off diagonal} entries gives
\[
q_1(AB_1)_{ij}+q_2(AB_2)_{ij}=0\qquad(i\ne j).
\]
The two quadrics are coprime, whereas the entries of $AB_1$ and $AB_2$ are linear. \newrev{Indeed, if $q_1L_1+q_2L_2=0$ with $L_1,L_2$ linear, then $q_1\mid L_2$ because $\gcd(q_1,q_2)=1$; the degree inequality forces $L_2=0$, and then $L_1=0$.} Hence all \chg{off diagonal} entries of $AB_1$ and $AB_2$ vanish. Comparing two diagonal entries in the same way shows that
\[
AB_1=\ell_1 I_3,\qquad AB_2=\ell_2 I_3
\]
for linear forms $\ell_1,\ell_2$.

\chg{Suppose $\ell_1$ is not the zero form. Choose a point $x\in\mathbb P^2$ with $\ell_1(x)\ne0$. Evaluating $A B_1=\ell_1I_3$ at $x$ shows that $B_1$ is invertible. The polynomial identity then gives $A=\ell_1B_1^{-1}$, so the coefficient matrices of $A$ span a one dimensional space, contradicting $\dim K=3$. Hence $\ell_1=0$, and the same argument gives $\ell_2=0$.}
The identity $A\operatorname{adj}A=(\det A)I_3$ now gives $\det A=0$. Using the adjugate identity on the other side yields
\[
0=(\operatorname{adj}A)A=q_1(B_1A)+q_2(B_2A).
\]
The same \chg{coprimeness and degree argument}, applied entrywise, gives $B_1A=B_2A=0$.

\chg{At least one of $B_1,B_2$ is nonzero, because otherwise $\operatorname{adj}A=0$, so all $2\times2$ minors would vanish identically, contradicting the choice of two independent quadrics $q_1,q_2$.} Choose a nonzero column $v$ and a nonzero row $\lambda$ of a nonzero $B_j$. Then
\[
Av=0,\qquad \lambda A=0.
\]
Thus $A$ has a common right kernel vector and a common left kernel covector. After independent changes of row and column bases, $A$ is supported in a $2\times2$ corner. Its restricted $2\times2$ minors therefore span a space of dimension at most one. Independent row and column changes preserve the dimension of this span, whereas the original span contains the independent quadrics $q_1,q_2$. This contradiction proves the bound.
\end{proof}

\chg{The comparison with general systems of plane quadrics can now be stated without adding an explanatory title to the formal result.}

\begin{remark}\label{rem:length-four}
Lemma~4.1 of \cite{GesmundoKeneshlou2026} treats \chg{zero dimensional} base loci of general linear systems of plane quadrics and includes a length four case. The preceding proof \newrev{gives a direct intrinsic reason} why this formal stratum disappears for the $2\times2$ minors of a genuine $3\times3$ linear matrix: the adjugate identity forces extra syzygies, and these syzygies exclude length four. \newrev{Thus the contribution here is the determinantal mechanism itself, rather than a claim that the absence of this stratum could not be inferred from existing tensor orbit classifications.} Equivalently, the \chg{length four quadratic base locus case} in the general collineation picture cannot be realized by a \chg{three dimensional} subspace of $M_3$ through its \chg{rank one} determinantal section.
\end{remark}

\chg{The remaining schemes can be realized by explicit linear matrix representatives.}

\begin{proposition}\label{prop:finite-representatives}
The eight types in Theorem~\ref{thm:finite-main} all occur. In homogeneous coordinates $[s:t:u]$ on the matrix net, the following representatives have the indicated \chg{rank one} schemes:
\[
\begin{array}{c|c}
\text{scheme type}&A(s,t,u)\\ \hline
\varnothing&\begin{pmatrix}0&s&t\\-s&0&u\\-t&-u&0\end{pmatrix}\\[2.0ex]
1\text{ point}&\begin{pmatrix}u&s&t\\0&0&s\\0&t&0\end{pmatrix}\\[2.0ex]
2\text{ reduced}&\begin{pmatrix}t&s&0\\-s&u&0\\0&0&s\end{pmatrix}\\[2.0ex]
\text{double point}&\begin{pmatrix}u&s&0\\s&t&0\\0&0&t\end{pmatrix}\\[2.0ex]
3\text{ reduced}&\operatorname{diag}(s,t,u)\\[1.0ex]
2+1&\begin{pmatrix}u&0&s\\0&t&0\\s&0&0\end{pmatrix}\\[2.0ex]
\text{curvilinear length }3&\begin{pmatrix}u&t&s\\t&s&0\\s&0&0\end{pmatrix}\\[2.0ex]
\mathfrak m^2&\begin{pmatrix}u&s&t\\s&0&0\\t&0&0\end{pmatrix}
\end{array}
\]
\end{proposition}

\begin{proof}
{Direct calculation gives the following ideals $I_{\min}(K)$ generated by the restricted minors. The last column records their saturations $I(Z_K)$.}
\[
\begin{array}{c|c|c}
{\text{scheme type}}&{I_{\min}(K)}&{I(Z_K)}\\ \hline
\varnothing&(s,t,u)^2&{(1)}\\
1\text{ point}&(su,s^2,tu,t^2,st)&{(s,t)}\\
2\text{ reduced}&(s^2,st,su,tu)&{(s,tu)}\\
\text{double point}&(s^2,st,t^2,tu)&{(t,s^2)}\\
3\text{ reduced}&(st,su,tu)&{(st,su,tu)}\\
2+1&(tu,st,s^2)&{(tu,st,s^2)}\\
\text{curvilinear }3&(su-t^2,st,s^2)&{(su-t^2,st,s^2)}\\
\mathfrak m^2&(s^2,st,t^2)&{(s^2,st,t^2)}.
\end{array}
\]
{
The saturated ideals have the asserted scheme structures. On the chart $u=1$, the double point has algebra $\mathbb C[s]/(s^2)$; the curvilinear triple has $s=t^2$ and $st=0$, hence algebra $\mathbb C[t]/(t^3)$; and the planar fat point has algebra $\mathbb C[s,t]/(s,t)^2$. For the $2+1$ configuration,
\[
(tu,st,s^2)=(t,s^2)\cap(s,u),
\]
so the double point at $[0:0:1]$ has tangent line $t=0$, whereas the reduced point $[0:1:0]$ lies off that line.

For completeness, let $Z\subset\mathbb P^2$ be finite of length at most three. Length two gives either two reduced points or a curvilinear double point. For length three, the support multiplicities are $1+1+1$, $2+1$, or $3$. In the last case, the local algebra $A$ has maximal ideal $\mathfrak m$ of vector space dimension two. If $\dim\mathfrak m/\mathfrak m^2=1$, then $A\simeq\mathbb C[\epsilon]/(\epsilon^3)$. If this dimension is two, then $\mathfrak m^2=0$ and $A\simeq\mathbb C[\epsilon,\eta]/(\epsilon,\eta)^2$; as a subscheme of the plane, this is the fat point defined by $\mathfrak m_p^2$.

For a curvilinear triple at $p$, choose affine linear coordinates $x,y$ vanishing at $p$ such that the image of $x$ generates $\mathfrak m$ and the image of $y$ lies in $\mathfrak m^2$. Thus $y=cx^2$ in $A$. If $c=0$, the scheme lies on the line $y=0$. If $c\ne0$, rescaling $y$ gives the local ideal $(y-x^2,x^3)$, whose saturated homogeneous ideal in coordinates $[s:t:u]=[y:x:1]$ is $(su-t^2,st,s^2)$. A double point and a distinct reduced point are contained in a line exactly when the reduced point lies on the tangent line of the double point; otherwise a projective change of coordinates gives the displayed $2+1$ configuration. Three noncollinear reduced points are projectively equivalent to the coordinate points. Consequently, every finite plane scheme of length at most three is projectively equivalent to a listed type or is a scheme of length three contained in a line. Only the latter alternative must be excluded.
} A length three scheme contained in a line cannot occur here: every restricted $2\times2$ minor is a quadric, and a quadratic on $\mathbb P^1$ vanishing on a length three subscheme vanishes identically, forcing that line to be a \chg{rank one component of positive dimension}. This yields exactly the stated list.
\end{proof}

\chg{Combining the positive dimensional and finite classifications gives the complete Hilbert polynomial list in $M_3$.}

\begin{corollary}\label{cor:hilbert-list}
For a \chg{three dimensional} $K\subset M_3$, the Hilbert polynomial of $Z_K=\mathbb P(K)\cap\Sigma_{3,3}$ is one of
\[
0,\quad1,\quad2,\quad3,\quad d+1,\quad d+2,\quad2d+1,\quad\binom{d+2}{2}.
\]
Every polynomial in the list occurs.
\end{corollary}

\section{Positive map realizations and exact separation}\label{sec:positive-maps}

\subsection{A \chg{diagonal type} obstruction}

Let $D_n\subset M_n$ be the diagonal subspace and $O_n$ the \chg{off diagonal} subspace, so $M_n=D_n\oplus O_n$. We say that a linear map $\Phi:M_n\to M_n$ has \chg{diagonal type} with nondegenerate \chg{off diagonal} part when
\[
\Phi(D_n)\subset D_n,\qquad
\Phi(O_n)\subset O_n,\qquad
\Phi|_{O_n}\text{ is invertible}.
\]

\begin{theorem}\label{thm:diagonal-obstruction}
If $\Phi:M_n\to M_n$ has \chg{diagonal type} with nondegenerate \chg{off diagonal} part, then
\[
K_1(\Phi)\subset\{[E_{11}],\ldots,[E_{nn}]\}.
\]
In particular $K_1(\Phi)$ is finite {and reduced}, possibly empty. The same conclusion holds for $\Phi^*$ with respect to the Hilbert--Schmidt pairing.
\end{theorem}

\begin{proof}
Write $X=D+O$. If $X\in\ker\Phi$, then $0=\Phi(D)+\Phi(O)$ with the two summands in complementary subspaces, so $\Phi(O)=0$ and hence $O=0$. Thus $\ker\Phi\subset D_n$. A nonzero diagonal matrix has \chg{rank one} exactly when one diagonal coordinate is nonzero. {At the scheme level, the Segre ideal restricts to $(x_ix_j:i<j)$ on $\mathbb P(D_n)$, the radical ideal of the coordinate points. Intersecting this reduced finite scheme with $\mathbb P(\ker\Phi)$ only retains the points contained in that linear space, so the intersection is reduced.} The statement for the adjoint follows because $D_n$ and $O_n$ are orthogonal and the adjoint of the invertible restriction to $O_n$ is invertible.
\end{proof}

\chg{The obstruction immediately separates smooth conic kernels from this large class of Choi type maps.}

\begin{corollary}\label{cor:diagonal-separation}
\newrev{No linear map $\Psi$ for which $K_1(\Psi)$ is a smooth conic spanning $\mathbb P(\ker\Psi)$ can be} \chg{determinantally equivalent at rank one} to a map of \chg{diagonal type} with nondegenerate \chg{off diagonal} part.
\end{corollary}

\begin{proof}
\newrev{By Theorem~\ref{thm:diagonal-obstruction}, the rank one kernel scheme of a map of diagonal type with nondegenerate off diagonal part is finite and reduced, whereas $K_1(\Psi)$ is a positive dimensional smooth conic. Proposition~\ref{prop:det-equivalence} shows that determinantally equivalent maps at rank one have projectively isomorphic rank one kernel schemes, which is impossible.}
\end{proof}

The architecture
\[
\Phi_A(X)=D_A(X)-X,
\]
where $D_A(X)$ is diagonal and depends only on the diagonal of $X$, is automatically of this type. This includes the classical generalized Choi constructions \cite{ChoKyeLee1992,ChruscinskiMarciniakRutkowski2018} and the doubly stochastic extension of \cite{Bera2025}.

\subsection{A positive family realizing the \chg{smooth conic} stratum}\label{subsec:positive-family}

Consider $\Psi_{p,q;r,t}$ from \eqref{eq:Psi}. The Miller and Olkiewicz map appears, after a simple diagonal output scaling, at the boundary point $(p,q,r,t)=(1,1,1,1)$; see \cite{MillerOlkiewicz2015,MarciniakRutkowski2017}. We use the \chg{four parameter} family as an exact realization of the geometric stratum; the construction itself is a \chg{scalar block} specialization of the merging framework and is not claimed as new.

To make the normalization inside the merging construction explicit, take $K_i=H_i=\mathbb C$ for $i=1,2$ in \cite[Definition~3.2 and (3.1)]{MarciniakRutkowski2017}, with
\[
\begin{gathered}
\phi_1(z)=\phi_2(z)=\omega_1(z)=\omega_2(z)=pz,\\
B_1=\frac r{\sqrt q},\qquad C_2=\frac t{\sqrt q},\qquad C_1=B_2=0.
\end{gathered}
\]
If $\phi$ denotes the resulting merging map and $D_q=\operatorname{diag}(1,1,\sqrt q)$, then direct substitution gives
\[
\Psi_{p,q;r,t}(X)=D_q\phi(X)D_q.
\]
For the Miller and Olkiewicz map $\phi_{\mathrm{MO}}$ in \cite[Example~3.4]{MarciniakRutkowski2017}, the corresponding output congruence is
\[
\Psi_{1,1;1,1}(X)=D_{\mathrm{MO}}\phi_{\mathrm{MO}}(X)D_{\mathrm{MO}},
\qquad
D_{\mathrm{MO}}=\operatorname{diag}(\sqrt2,\sqrt2,1).
\]

\chg{The following sum of squares identity gives the exact positivity region directly.}

\begin{theorem}\label{thm:exact-positivity}
For $p,q>0$ and $r,t\ge0$, the map $\Psi_{p,q;r,t}$ is positive if and only if
\[
r^2\le pq,\qquad t^2\le pq.
\]
More precisely, for $x=(a,b,c)^{\mathsf T}$, $y=(u,v,w)^{\mathsf T}$, and $\sigma=|a|^2+|b|^2$,
\begin{align}
 p\sigma\,y^*\Psi_{p,q;r,t}(xx^*)y
 &=|p\sigma u+r a\bar c\,w|^2+|p\sigma v+t c\bar b\,w|^2 \label{eq:sos-positivity}\\
 &\quad +(pq-r^2)|acw|^2+(pq-t^2)|bcw|^2.\nonumber
\end{align}
\end{theorem}

\begin{proof}
Expanding the right hand side gives
\[
p^2\sigma^2(|u|^2+|v|^2)+pq\sigma|c|^2|w|^2
+2pr\sigma\Re(\bar u\,a\bar c\,w)
+2pt\sigma\Re(\bar v\,c\bar b\,w),
\]
which is exactly $p\sigma\,y^*\Psi(xx^*)y$. Under the two inequalities the expression is nonnegative; the case $\sigma=0$ is immediate. Conversely, setting $b=0$ and examining the $(1,3)$ principal block of $\Psi(xx^*)$ forces $r^2\le pq$, while setting $a=0$ forces $t^2\le pq$.
\end{proof}

\chg{For {$p,q,r,t>0$}, the kernel does not vary with the parameters and is the normal form identified in Section~\ref{sec:geometry}.}

\begin{proposition}\label{prop:kernel}
If $p,q,r,t>0$, then
\[
\ker\Psi_{p,q;r,t}=K_0=
\left\{
\begin{pmatrix}
z&x&0\\
y&-z&0\\
0&0&0
\end{pmatrix}:x,y,z\in\mathbb C
\right\}.
\]
Consequently its \chg{rank one} kernel scheme is the smooth conic $z^2+xy=0$.
\end{proposition}

\begin{proof}
From $\Psi(X)=0$ one obtains
\[
x_{13}=x_{31}=x_{23}=x_{32}=x_{33}=0,
\qquad x_{11}+x_{22}=0,
\]
while $x_{12},x_{21}$ are free. The determinant of the nonzero $2\times2$ corner is $-z^2-xy$.
\end{proof}

{
The restriction $r,t>0$ is essential for the matrix net conclusion. For $p,q>0$ and $r,t\ge0$, the same kernel equations give
\[
\dim\ker\Psi_{p,q;r,t}
=3+2\mathbf 1_{\{r=0\}}+2\mathbf 1_{\{t=0\}}.
\]
Indeed, setting $r=0$ frees the entries $x_{13},x_{31}$, and setting $t=0$ frees $x_{23},x_{32}$. Thus the kernel has dimension five on either coordinate axis away from the origin, and dimension seven at the origin; these boundary kernels are not matrix nets.
}

\chg{An output congruence removes the dimensional scales $p$ and $q$.}

\begin{proposition}\label{prop:normalization}
Let
\[
R=\frac r{\sqrt{pq}},\qquad T=\frac t{\sqrt{pq}},\qquad
D=\operatorname{diag}(p^{-1/2},p^{-1/2},q^{-1/2}).
\]
Then
\[
\Ad_D\circ\Psi_{p,q;r,t}=\Psi_{1,1;R,T},
\qquad
\Ad_D(Y)=DYD.
\]
Consequently positivity, $k$-positivity, complete positivity, $k$-copositivity, complete copositivity, decomposability, and atomicity are preserved and reflected by this normalization. The kernel is unchanged. Thus, modulo invertible positive output congruence, the \chg{four parameter} family reduces to the two dimensionless parameters $(R,T)$.
\end{proposition}

\begin{proof}
A direct calculation gives
\[
D\Psi_{p,q;r,t}(X)D=
\begin{pmatrix}
x_{11}+x_{22}&0&R x_{13}\\
0&x_{11}+x_{22}&T x_{32}\\
R x_{31}&T x_{23}&x_{33}
\end{pmatrix}.
\]
Both $\Ad_D$ and $\Ad_{D^{-1}}$ are completely positive order isomorphisms. Their ampliations are again invertible congruences, so they preserve and reflect $k$-positivity and complete positivity; the same statement after partial transpose gives the copositive assertions. They also carry CP$+$coCP decompositions, and more generally sums of $2$-positive and $2$-copositive maps, bijectively to such decompositions. Hence decomposability and atomicity are preserved and reflected. Since $D$ is invertible, $\ker(\Ad_D\circ\Psi)=\ker\Psi$.
\end{proof}

{We record the normalized phase picture here. Positivity follows from Proposition~\ref{prop:normalization} and Theorem~\ref{thm:exact-positivity}; the decomposability, atomicity, and edge assertions follow from Theorem~\ref{thm:exact-decomposability} and Corollary~\ref{cor:2positive} below.}

\begin{corollary}\label{cor:normalized-phase}
In the normalized first quadrant $(R,T)$, positivity is exactly the unit square
\[
0\le R\le1,\qquad0\le T\le1,
\]
decomposability is exactly the quarter disk
\[
R^2+T^2\le1,
\]
and the complementary part of the positivity square is atomic. \chg{The edge $T=0$ consists exactly of the completely positive maps in this family, while the edge $R=0$ consists exactly of the completely copositive maps.}
\end{corollary}

\subsection{Exact decomposability and atomicity phase diagram}

{
We use the conventions
\[
C_\Phi=\sum_{i,j}E_{ij}\otimes\Phi(E_{ij}),
\qquad
\Gamma=\operatorname{id}\otimes\mathsf T,
\]
where $\mathsf T$ is transposition in the standard basis. Thus $\Phi$ is $k$-copositive when $\mathsf T\circ\Phi$ is $k$-positive. A map is decomposable if it is a sum of a completely positive and a completely copositive map. A positive map is called atomic if it cannot be written as a sum of a $2$-positive and a $2$-copositive map.
}

Write $C_\Psi$ for the Choi matrix in the ordered basis
\[
(f_{11},f_{12},f_{13},f_{21},f_{22},f_{23},f_{31},f_{32},f_{33}),
\qquad f_{ij}=e_i\otimes e_j.
\]
Then
\[
C_\Psi=
\begin{pmatrix}
p&0&0&0&0&0&0&0&r\\
0&p&0&0&0&0&0&0&0\\
0&0&0&0&0&0&0&0&0\\
0&0&0&p&0&0&0&0&0\\
0&0&0&0&p&0&0&0&0\\
0&0&0&0&0&0&0&t&0\\
0&0&0&0&0&0&0&0&0\\
0&0&0&0&0&t&0&0&0\\
r&0&0&0&0&0&0&0&q
\end{pmatrix}.
\]

\chg{The decomposable cone meets the family along exactly one quadratic boundary.}

\begin{theorem}\label{thm:exact-decomposability}
Assume $\Psi_{p,q;r,t}$ is positive. Then
\[
\Psi_{p,q;r,t}\text{ is decomposable}
\quad\Longleftrightarrow\quad
r^2+t^2\le pq.
\]
If $r^2+t^2>pq$, then $\Psi_{p,q;r,t}$ is atomic.
\end{theorem}

\begin{proof}
Assume first $r^2+t^2\le pq$. Put
\[
v=\sqrt p\,f_{22}+\frac t{\sqrt p}f_{33},
\qquad
Q=|v\rangle\langle v|\succeq0.
\]
Partial transpose on the second tensor factor converts the cross term in $Q$ into the $t$ block on $\Span\{f_{23},f_{32}\}$. Hence
\[
P:=C_\Psi-Q^\Gamma
\]
is the direct sum of the scalar $p$ blocks at $f_{12}$ and $f_{21}$, zero blocks, and the nontrivial block
\[
\begin{pmatrix}
p&r\\ r&q-t^2/p
\end{pmatrix}
\]
on $\Span\{f_{11},f_{33}\}$. {Here $p>0$ and $q-t^2/p\ge r^2/p\ge0$, while the determinant is $pq-r^2-t^2\ge0$. Hence $P\succeq0$, including the equality case $r^2+t^2=pq$.} Thus
\[
C_\Psi=P+Q^\Gamma
\]
is a completely positive plus completely copositive decomposition.

For the converse, let
\[
\alpha=\frac rp,\qquad \beta=\frac tp
\]
and define
\begin{align}
\rho_{p,r,t}
&=|\alpha f_{11}-f_{33}\rangle\langle\alpha f_{11}-f_{33}| \label{eq:rho}\\
&\quad+\beta|f_{23}-f_{32}\rangle\langle f_{23}-f_{32}|\nonumber\\
&\quad+\alpha\bigl(|f_{13}\rangle\langle f_{13}|+|f_{31}\rangle\langle f_{31}|\bigr)
+\beta^2|f_{22}\rangle\langle f_{22}|.\nonumber
\end{align}
Clearly $\rho_{p,r,t}\succeq0$. Its partial transpose can be regrouped as
\begin{align}
\rho_{p,r,t}^\Gamma
&=\alpha|f_{13}-f_{31}\rangle\langle f_{13}-f_{31}| \label{eq:rhoGamma}\\
&\quad+|\beta f_{22}-f_{33}\rangle\langle\beta f_{22}-f_{33}|\nonumber\\
&\quad+\alpha^2|f_{11}\rangle\langle f_{11}|
+\beta\bigl(|f_{23}\rangle\langle f_{23}|+|f_{32}\rangle\langle f_{32}|\bigr),\nonumber
\end{align}
so $\rho_{p,r,t}$ is PPT.\@ Direct evaluation gives
\begin{equation}\label{eq:pairing}
\Tr(C_\Psi\rho_{p,r,t})
=q-\frac{r^2+t^2}{p}
=\frac{pq-r^2-t^2}{p}.
\end{equation}
Therefore $r^2+t^2>pq$ is incompatible with decomposability.

More is true. Every vector in the decomposition \eqref{eq:rho} has Schmidt rank at most two, so $\rho_{p,r,t}$ has Schmidt number at most two. The decomposition \eqref{eq:rhoGamma} shows the same for $\rho_{p,r,t}^\Gamma$. \newrev{Suppose that $\Psi=\Phi_1+\Phi_2$, where $\Phi_1$ is $2$-positive and $\Phi_2$ is $2$-copositive. Then $C_{\Phi_1}$ and $C_{\Phi_2}^\Gamma$ are $2$-block-positive. Since $\SN(\rho_{p,r,t})\le2$ and $\SN(\rho_{p,r,t}^\Gamma)\le2$, the Choi pairing satisfies}
\[
\newrev{\Tr(C_\Psi\rho_{p,r,t})=\Tr(C_{\Phi_1}\rho_{p,r,t})+\Tr(C_{\Phi_2}^\Gamma\rho_{p,r,t}^\Gamma)\ge0.}
\]
\newrev{This contradicts \eqref{eq:pairing} when $r^2+t^2>pq$.} Hence the map is atomic.
\end{proof}

{
At $r^2+t^2=pq$, the pairing in \eqref{eq:pairing} is zero, whereas decomposability follows from the explicit decomposition $C_\Psi=P+Q^\Gamma$ above. Vanishing of this single pairing does not imply separability of $\rho_{p,r,t}$.
}

\chg{The separating operator in the proof is itself an entangled state with controlled ranks and Schmidt numbers.}

\begin{corollary}\label{cor:adapted-PPT}
Assume $\Psi_{p,q;r,t}$ is positive and $r^2+t^2>pq$. Then $r,t>0$, and the normalized operator
\[
\widehat\rho_{p,r,t}=\frac{\rho_{p,r,t}}{\Tr\rho_{p,r,t}}
\]
is a PPT entangled state satisfying
\[
\rank\widehat\rho_{p,r,t}=\rank\widehat\rho_{p,r,t}^\Gamma=5,
\qquad
\SN(\widehat\rho_{p,r,t})=\SN(\widehat\rho_{p,r,t}^\Gamma)=2.
\]
In particular the atomic region carries an explicit family of PPT entangled states of birank $(5,5)$ and \chg{Schmidt numbers $(2,2)$ for the state and its partial transpose}. If $p,r,t$ are rational, then $\widehat\rho_{p,r,t}$ has rational matrix entries.
\end{corollary}

\begin{proof}
Equations~\eqref{eq:rho}--\eqref{eq:rhoGamma} show that $\rho_{p,r,t}$ and its partial transpose are positive. Equation~\eqref{eq:pairing} is negative in the atomic region, and $C_\Psi$ is block positive because $\Psi$ is positive. Hence $\rho_{p,r,t}$ is entangled. Separability is preserved by partial transpose, so $\rho_{p,r,t}^\Gamma$ is entangled as well.

The five nonzero vectors appearing in \eqref{eq:rho} have mutually independent supports, and the same is true for the five vectors in \eqref{eq:rhoGamma}; therefore both operators have rank five. Those decompositions use vectors of Schmidt rank at most two, so both Schmidt numbers are at most two. Entanglement makes them at least two, hence they are exactly two. Finally, for rational $p,r,t$ the coefficients $\alpha=r/p$ and $\beta=t/p$ are rational, and normalization divides by a positive rational trace.
\end{proof}

\chg{Higher positivity collapses to complete positivity on the coordinate edges of this family.}

\begin{corollary}\label{cor:2positive}
Inside the positivity region,
\[
\Psi_{p,q;r,t}\text{ is $2$-positive}
\Longleftrightarrow
\Psi_{p,q;r,t}\text{ is completely positive}
\Longleftrightarrow t=0,
\]
and
\[
\Psi_{p,q;r,t}\text{ is $2$-copositive}
\Longleftrightarrow
\Psi_{p,q;r,t}\text{ is completely copositive}
\Longleftrightarrow r=0.
\]
\end{corollary}

\begin{proof}
If $t=0$, positivity gives $r^2\le pq$, and the displayed Choi matrix is positive semidefinite, so the map is completely positive. If $t>0$, the \chg{Schmidt rank two} vector $f_{23}-f_{32}$ satisfies
\[
\langle f_{23}-f_{32},C_\Psi(f_{23}-f_{32})\rangle=-2t<0,
\]
so the map is not $2$-positive. The copositive assertion follows by partial transpose, using $f_{13}-f_{31}$ and the parameter $r$.
\end{proof}

The general merging construction of \cite{MarciniakRutkowski2017} already supplies broad structural criteria for complete positivity, $2$-positivity, decomposability, and nondecomposability. Thus the point of Corollary~\ref{cor:2positive} is not a new universal collapse phenomenon; for the scalar slice \eqref{eq:Psi}, the contribution is the transparent direct proof together with the exact matching decomposability boundary in Theorem~\ref{thm:exact-decomposability}.

There are several useful recent points of comparison. Ho, Le, Le, and Osaka characterize decomposable Choi polynomials through the Gram cone $C_\Phi=Q+R^\Gamma$ with $Q,R\succeq0$ \cite{HoLeLeOsaka2026}; the first half of the proof of Theorem~\ref{thm:exact-decomposability} is an explicit \chg{single block Gram cone} certificate of precisely this form, while \eqref{eq:sos-positivity} is the complementary \chg{block positivity} certificate. \chg{Poderini et al.\ recently exhibited other sparse qutrit families with exact positivity and nondecomposability thresholds and PPT entanglement transitions \cite{PoderiniEtAl2026}.} \newrev{Sacchi subsequently obtained an anisotropic bistochastic two parameter family with an exact square positivity region, a circular decomposability boundary, and explicit PPT entangled states, including rank type $(5,5)$ at a distinguished endpoint \cite{Sacchi2026}. These comparisons make it important to separate phase diagram novelty from geometric novelty. In the present paper the map family is used as a realization of a previously classified three dimensional domain kernel stratum, namely the smooth conic; the exact boundary is therefore attached to the projective determinantal geometry rather than derived from sparsity alone.}

\subsection{A fixed integral PPT certificate}

The adapted witness gives the exact boundary, but a single PPT state \chg{independent of the parameters} remains useful as an arithmetic certificate. Let
\begin{align}
\rho&=|f_{11}-f_{33}\rangle\langle f_{11}-f_{33}|
+|f_{23}-f_{32}\rangle\langle f_{23}-f_{32}|\label{eq:rho-fixed}\\
&\quad+|f_{13}\rangle\langle f_{13}|+|f_{22}\rangle\langle f_{22}|+|f_{31}\rangle\langle f_{31}|.\nonumber
\end{align}
Then
\[
\begin{aligned}
\rho^\Gamma
&=|f_{13}-f_{31}\rangle\langle f_{13}-f_{31}|
+|f_{22}-f_{33}\rangle\langle f_{22}-f_{33}|\\
&\quad+|f_{11}\rangle\langle f_{11}|+|f_{23}\rangle\langle f_{23}|+|f_{32}\rangle\langle f_{32}|.
\end{aligned}
\]
so $\rho$ is PPT and $\widehat\rho=\rho/7$ is a rational PPT state.

\begin{proposition}\label{prop:fixed-certificate}
For all parameters,
\[
\Tr(C_\Psi\rho)=2p+q-2(r+t).
\]
Consequently, within the positivity region,
\[
2(r+t)>2p+q
\]
is a \chg{sufficient atomicity certificate furnished by the fixed PPT state $\rho$}, and
\[
\Tr(C_\Psi\widehat\rho)=\frac{2p+q-2(r+t)}7.
\]
\end{proposition}

\begin{proof}
Evaluating $C_\Psi$ on the five \chg{rank one} summands in \eqref{eq:rho-fixed} gives respectively
\[
p+q-2r,\qquad -2t,\qquad0,\qquad p,\qquad0.
\]
The sum is the asserted affine expression. The atomic conclusion follows either from Theorem~\ref{thm:exact-decomposability} or directly from the \chg{Schmidt rank two} decompositions of $\rho$ and $\rho^\Gamma$.
\end{proof}

At the rational point
\[
(p,q,r,t)=\left(1,1,\frac45,\frac45\right),
\]
one has $r^2=t^2=16/25<1$ and
\[
\Tr(C_\Psi\widehat\rho)=-\frac1{35}.
\]
The Hermitian certificate \eqref{eq:sos-positivity} is defined over $\mathbb Q(i)$ using rational square coefficients because
\[
1-(4/5)^2=(3/5)^2.
\]

\subsection{Separation from the Miller and Olkiewicz boundary orbit}

\begin{proposition}\label{prop:rank-separation}
Assume $p,q,r,t>0$ and $r^2<pq$, $t^2<pq$. Then
\[
\rank C_{\Psi_{p,q;r,t}}=7,
\qquad
\rank C_{\Psi_{p,q;r,t}}^\Gamma=7.
\]
At $(p,q,r,t)=(1,1,1,1)$ both ranks equal six.
\end{proposition}

\begin{proof}
The Choi matrix has three scalar $p$ blocks, one block
\[
\begin{pmatrix}p&r\\r&q\end{pmatrix},
\]
one \chg{off diagonal} block
\[
\begin{pmatrix}0&t\\t&0\end{pmatrix},
\]
and two zero coordinates, giving rank seven in the strict region and rank six at the boundary. After partial transpose the roles of $r$ and $t$ in the two nontrivial block types are interchanged, giving the same conclusion.
\end{proof}

The unordered pair $\{\rank C_\Phi,\rank C_\Phi^\Gamma\}$ is preserved by invertible local congruences and merely interchanged by input or output transposition. Hence strict points are not in the \chg{orbit obtained from the Miller and Olkiewicz boundary representative by local filtering and transposition}.

\section{Consequences for Choi constructions and related questions}\label{sec:consequences}

The geometric theory above is independent of positive map terminology. Generalized Choi maps in the constructions of Cho, Kye, and Lee and of Chru\'sci\'nski, Marciniak, and Rutkowski have diagonal output depending on the diagonal of $X$ and act nondegenerately on \chg{off diagonal} matrix units \cite{ChoKyeLee1992,ChruscinskiMarciniakRutkowski2018}. The later doubly stochastic family of Bera, Scala, Sarbicki, and Chru\'sci\'nski has the same relevant architecture \cite{Bera2025}; recent optimality results for generalized Choi maps appear in \cite{ScalaEtAl2024}. \chg{Osaka's extremal family \cite{Osaka1992} belongs to the same historical line of positive map constructions, but no application of Theorem~\ref{thm:diagonal-obstruction} to that family is needed here.} Hence Theorem~\ref{thm:diagonal-obstruction} gives a \chg{parameter free} separation for the \chg{diagonal type} families just described.

\begin{corollary}\label{cor:choi-separation}
\chg{Every $\Psi_{p,q;r,t}$ with {$p,q,r,t>0$} lies outside the rank one determinantal equivalence class of every map of diagonal type with nondegenerate off diagonal part. In particular, it is not determinantally equivalent to the generalized Choi families cited above.}
\end{corollary}

\begin{proof}
\newrev{For $p,q,r,t>0$, Proposition~\ref{prop:kernel} identifies $K_1(\Psi_{p,q;r,t})$ with a smooth conic spanning $\mathbb P(\ker\Psi_{p,q;r,t})$. The first assertion follows from Corollary~\ref{cor:diagonal-separation}. The generalized Choi families cited in the preceding paragraph satisfy the diagonal type hypothesis, so the final assertion follows as well.}
\end{proof}

Osaka's March 2026 VIASM lectures on detecting entanglement by positive linear maps provide background for this discussion \cite{OsakaLectures2026}. Here we compare maps through explicitly defined equivalence relations. Formula~\eqref{eq:Psi} is itself a specialization of the merging construction \cite{MarciniakRutkowski2017} followed by the positive diagonal output congruence described above. The precise separation results are the determinantal separation in Corollary~\ref{cor:choi-separation} and the \chg{separation under local congruence and transposition} in Proposition~\ref{prop:rank-separation}. \newrev{This fixes the scope of our comparison with the related positive map and entanglement witness constructions studied in} \cite{ChoKyeLee1992,Ha1998,SenguptaArvind2011,SenguptaArvind2013}.

Two related questions lie beyond the results of this paper. The question whether every $2$-positive map on $M_4(\mathbb C)$ is decomposable extends the $M_3(\mathbb C)$ theorem of Yang, Leung, and Tang \cite{YangLeungTang2016}. Our $M_4$ Grassmannian result concerns \chg{rank one} geometry of \chg{three dimensional} domain kernels, whereas $2$-positivity is a \chg{condition involving vectors of Schmidt rank at most two} on the Choi matrix; moreover Corollary~\ref{cor:2positive} shows that the family \eqref{eq:Psi} cannot produce a counterexample. Likewise, a \chg{classification based on ranges} of PPT edge states in $M_4(\mathbb C)\otimes M_4(\mathbb C)$, motivated by the $3\otimes3$ classification in \cite{KyeOsaka2012}, requires compatible information for a state and its partial transpose. Choi polynomial methods give a complementary route to PPT entanglement and edge state questions \cite{HoLeLeOsaka2026}. The present work classifies one kernel scheme at a time and therefore supplies geometric input rather than such a classification.

Recent \chg{exact threshold} results for sparse qutrit maps \chg{\cite{PoderiniEtAl2026}} \newrev{and the closely related anisotropic phase geometry in \cite{Sacchi2026}} make this distinction useful: sparse Choi structure can itself lead to sharp semialgebraic phase boundaries. \newrev{What is specific here is that an exact boundary is placed on the smooth conic component singled out in advance by the determinantal classification.} In normalized coordinates $(R,T)$, the resulting picture is the unit positivity square cut by the decomposability quarter circle $R^2+T^2=1$. Corollary~\ref{cor:adapted-PPT} further shows that the atomic side of this geometric boundary comes equipped with a canonical \chg{PPT entangled family of birank $(5,5)$}.

The fixed rational certificate also gives a uniform construction of PPT entangled states on $\mathbb C^m\otimes\mathbb C^n$ for $n\ge m\ge3$.

\begin{corollary}\label{cor:full-rank-rational}
Let $n\ge m\ge3$. There exists a \chg{full rank} PPT entangled state on $\mathbb C^m\otimes\mathbb C^n$ with rational matrix entries. Let $W=C_{\Psi_{1,1;4/5,4/5}}$, let $\rho$ be the integral PPT operator in \eqref{eq:rho-fixed}, choose rational $0<\delta<1/25$ and rational $\eta>0$, and let
\[
J:\mathbb C^3\otimes\mathbb C^3\longrightarrow\mathbb C^m\otimes\mathbb C^n
\]
be the standard coordinate isometry. Then
\[
R_{m,n}=J(\rho+\delta I_9)J^*+\eta(I_{mn}-JJ^*)
\]
is positive definite, $R_{m,n}^\Gamma$ is positive definite, and $R_{m,n}/\Tr R_{m,n}$ is entangled.
\end{corollary}

\begin{proof}
By Proposition~\ref{prop:fixed-certificate}, $\Tr(W\rho)=-1/5$ and $\Tr W=5$, so
\[
\Tr\bigl(W(\rho+\delta I_9)\bigr)=-\frac15+5\delta<0.
\]
The operators $\rho+\delta I_9$ and its partial transpose are positive definite. Write $J=J_A\otimes J_B$, where $J_A,J_B$ are the standard coordinate isometries. Their entries are real, so
\[
(JXJ^*)^\Gamma=JX^\Gamma J^*
\]
for every $X\in M_9$. For the standard coordinate embedding, $JJ^*$ is a sum of product coordinate projections and is fixed by partial transpose; hence $R_{m,n}$ and $R_{m,n}^\Gamma$ are positive definite. \newrev{Moreover, for every product vector $x\otimes y$,}
\[
\newrev{\langle x\otimes y,JWJ^*(x\otimes y)\rangle=\langle J_A^*x\otimes J_B^*y,W(J_A^*x\otimes J_B^*y)\rangle\ge0,}
\]
\newrev{so $JWJ^*$ is block positive.} Therefore
\[
\Tr(JWJ^*R_{m,n})=\Tr\bigl(W(\rho+\delta I_9)\bigr)<0.
\]
Thus $R_{m,n}$ is entangled. Rational $\delta$ and $\eta$ give rational entries.
\end{proof}

This construction uses a $3\times3$ corner embedding with positive padding. It gives explicit \chg{full rank} PPT entangled states with rational entries in every pair of dimensions $n\ge m\ge3$. \chg{This is an explicit rational construction rather than a claim that the existence of full rank PPT entangled states is new.} The main geometric results and the phase diagram of Section~\ref{sec:positive-maps} are independent of this construction. For future realizability questions under stronger positivity assumptions, the recent characterizations and generation methods in \cite{vomEndeKhatriDenisov2025} provide a complementary \chg{operator theoretic} viewpoint.

\medskip
\noindent{\textbf{Conclusion.}
The paper establishes a geometric classification of \chg{rank one} determinantal kernel schemes for matrix nets and connects one of the resulting strata with an exact positive map phase diagram. In arbitrary rectangular size, every \chg{positive dimensional} \chg{rank one} section is one of the six schemes in Theorem~\ref{thm:main-positive-dimensional}, and for $a,b\ge3$ these sections form exactly three irreducible Grassmannian components with explicitly determined dimensions and intersections. In $M_3$, the adjugate identity gives the sharp bound $\operatorname{length}Z_K\le3$ and leads to a complete list of finite scheme types and Hilbert polynomials. On the operator side, the merging slice $\Psi_{p,q;r,t}$ realizes the smooth conic stratum {for $r,t>0$} and has exact positivity, decomposability, and atomicity boundaries; the atomic region is separated by explicit PPT states with controlled birank and Schmidt number. Natural next questions are to study higher determinantal kernel profiles $K_r(L)$, matrix spaces of dimension greater than three, and the finite \chg{rank one} geometry of nets in larger square matrix spaces. A second direction is to determine which determinantal strata can occur under $k$-positivity or other positivity constraints, where the relevant Choi geometry involves Schmidt rank at most $k$ rather than \chg{rank one} alone. These problems may provide a bridge between the projective geometry of determinantal sections and structural questions on decomposable, atomic, and higher positive maps. \chg{The comparison with the recent sparse qutrit phase diagrams in \newrev{\cite{PoderiniEtAl2026,Sacchi2026}} also suggests asking which exact semialgebraic boundary phenomena are forced by a determinantal kernel stratum and which arise from sparsity alone.}}

\end{document}